\documentclass[11pt]{article}
\usepackage{amsmath, amscd, amssymb, latexsym, epsfig, color, amsthm}
\usepackage[all]{xy}
\numberwithin{equation}{section}

\newtheorem{theorem}{Theorem}[section]

\newtheorem{lemma}[theorem]{Lemma}

\theoremstyle{definition}

\newtheorem{remark}[theorem]{Remark}

\DeclareMathOperator\lk{\mathrm{lk}}
\DeclareMathOperator\cost{\mathrm{cost}}
\DeclareMathOperator\st{\mathrm{st}}

\newcommand{\BM}{BM}
\newcommand{\ST}{\mathcal{ST}}
\newcommand{\field}{{\bf k}}

\newcommand{\R}{{\mathbb R}}
\newcommand{\Q}{{\mathbb Q}}
\newcommand{\Z}{{\mathbb Z}}
\newcommand{\Sp}{{\mathbb S}}

\title{Minimal Balanced Triangulations of Sphere Bundles over the Circle}
\author{Hailun Zheng\\
	\small Department of Mathematics \\[-0.8ex]
	\small University of Washington\\[-0.8ex]
	\small Seattle, WA 98195-4350, USA\\[-0.8ex]
	\small \texttt{hailunz@math.washington.edu}
}
\begin{document}
	\maketitle
	\begin{abstract}
		We determine the minimum number of vertices needed to provide balanced triangulations of $\Sp^{d-2}$-bundles over $\Sp^1$. If $d$ is odd and the bundle is orientable, or  $d$ is even and the bundle is non-orientable, the minimum number of vertices is $3d$; otherwise, it is $3d+2$. Similar results apply to all balanced simplicial complexes that triangulate homology manifolds with $\beta_1\neq 0$ and $\beta_2=0$, where $\beta_i$'s are the Betti numbers, computed with coefficients in $\Q$.
	\end{abstract}		
	
	\section{Introduction}
	
	What is the minimum number of vertices needed to construct a triangulation of $\Sp^{d-2}\times \Sp^1$ or of the non-orientable $\Sp^{d-2}$-bundle over $\Sp^1$? This question was first studied by K$\rm\ddot{u}$hnel in \cite{K} for PL-triangulations, where he gave a construction with $2d+1$ vertices. Later Bagchi and Datta \cite{BD} proved, in the context of topological triangulations, that any non-simply connected $(d-1)$-dimensional closed manifold requires at least $2d+1$ vertices, and if it has $2d+1$ vertices, then it is isomorphic to one of K$\rm\ddot{u}$hnel's minimal triangulations. In the same year, Chestnut, Sapir and Swartz \cite{CSS} established a similar result. In fact, they characterized all pairs $(f_0,f_1)$, where $f_0$ is the number of vertices and $f_1$ is the number of edges, that are possible for triangulations of $\Sp^{d-2}$-bundles over $\Sp^1$. Both papers \cite{BD} and \cite{CSS} showed that if $d$ is odd and the bundle is orientable, or if $d$ is even and the bundle is non-orientable, then the minimum number of vertices needed is $2d+1$, while in the two other cases, the minimum is $2d+2$.
	
	It is natural to ask the same question for the case of balanced triangulations. In \cite{KN}, Klee and Novik gave an explicit construction of a $3d$-vertex balanced simplicial complex whose geometric realization is a sphere bundle over the circle (orientable or non-orientable depending on the parity of $d$). They also described similar constructions with any number $n\geq 3d+2$ of vertices that provide triangulations of both orientable and non-orientable $\Sp^{d-2}$-bundles over $\Sp^1$. However, they left open the question whether such $3d$-vertex construction is unique, and whether there exists a $(3d+1)$-vertex triangulation. 
	
	In this paper, we answer these two questions by providing an affirmative answer to the conjecture raised in \cite[Conjecture 6.8]{KN}. We show that the construction of balanced $3d$-vertex triangulation in \cite{KN} is unique in the category of homology $(d-1)$-manifolds with $\beta_1\neq 0$ and $\beta_2=0$, where Betti numbers are computed with coefficients in $\Q$. In particular, it applies to all $\Sp^{d-2}$-bundles over $\Sp^1$ for $d>4$; and in the case $d=4$, only the non-orientable $\Sp^2$-bundle is relevant, where in fact $\beta_2=0$. Besides that, we also show that there exist no balanced $(3d+1)$-vertex triangulations of $\Sp^{d-2}$-bundles over $\Sp^1$. 
	
	The paper is structured as follows. In Section 2, we review the definitions and basic facts that will be necessary for our proofs. In Section 3, we establish the uniqueness of the balanced $3d$-vertex construction, see Theorem \ref{thm 3d}. In Section 4,  we verify that no balanced $(3d+1)$-vertex triangulation exists, see Theorem \ref{thm 3d+1}.  
	
	\section{Preliminaries}
	
	A \textit{simplicial complex} $\Delta$ on vertex set $V$ is a collection of subsets
	$\sigma\subseteq V$, called \textit{faces}, that is closed under inclusion, and such that for every $v \in V$, $\{v\} \in \Delta$. For $\sigma\in \Delta$, let $\dim\sigma:=|\sigma|-1$ and define the \textit{dimension} of $\Delta$, $\dim \Delta$, as the maximum dimension of the faces of $\Delta$. The \textit{facets} of $\Delta$ are maximal under inclusion faces of $\Delta$. We say that a simplicial complex $\Delta$ is \textit{pure} if all of its facets have the same dimension.
	
	We let $d=\dim\Delta +1$ throughout. For $-1 \leq i \leq d-1$, the \textit{$f$-number} $f_i = f_i(\Delta)$ denotes the number of $i$-dimensional faces of $\Delta$. It is often more convenient to study the \textit{$h$-numbers} $h_i=h_i(\Delta)$, $0 \leq i \leq d$, defined by the relation $\sum_{j=0}^{d}h_j\lambda^{d-j}=\sum_{i=0}^{d}f_{i-1}(\lambda-1)^{d-i}$.
	
	If $\Delta$ is a simplicial complex and $\sigma$ is a face of $\Delta$, the \textit{star} of $\sigma$ in $\Delta$ is $\st_\Delta \sigma:= \{\tau \in\Delta: \sigma\cup\tau\in\Delta \}$, and the \textit{contrastar} of $\sigma$ in $\Delta$ is $\cost_\Delta \sigma:=\{\tau \in \Delta:\sigma \not\subseteq\tau\}$. We also define the \textit{link} of $\sigma$ in $\Delta$ as $\lk_\Delta \sigma:=\{\tau-\sigma\in \Delta: \sigma\subseteq \tau\in \Delta\}$, the \textit{deletion} of a subset of vertices $W$ from $\Delta$ as $\Delta\backslash W:=\{\sigma\in\Delta:\sigma\cap W=\emptyset\}$, and the \emph{restriction} of $\Delta$ to a vertex set $W$ as $\Delta[W]:=\{\sigma\in \Delta:\sigma\subseteq W\}$. Finally, we recall that $F\subseteq V$ is a \textit{missing face} if $F\notin \Delta$ but all proper subsets of $F$ are faces of $\Delta$; $F$ is a \textit{missing $k$-face} if it is a missing face and $|F|=k+1$.
	
	A $(d-1)$-dimensional simplicial complex $\Delta$ is called \textit{balanced} if the graph of $\Delta$ is $d$-colorable, or equivalently, there is a coloring $\kappa:
	V(\Delta) \to [d]$, with $[d] = \{1,\cdots, d\}$, such that $\kappa(u) ̸\neq \kappa(v)$ for all edges $\{u, v\} \in\Delta$. The \emph{$S$-rank-selected subcomplex} of $\Delta$ is defined as $\Delta_S := \{\tau \in\Delta : \kappa(\tau) \subseteq S\}$ for $S\subseteq [d]$.
	
	A simplicial complex $\Delta$ is a \textit{simplicial manifold} if the geometric realization of $\Delta$ is homeomorphic to a manifold. We denote by $\tilde{H}_*(\Delta;\field)$ the reduced homology with coefficients in a field $\field$, and denote the reduced Betti numbers of $\Delta$ with coefficients in $\field$ by $\beta_i(\Delta;\field):=\dim_{\field}\tilde{H}_i(\Delta;\field)$. We say that $\Delta$ is a $(d-1)$-dimensional \textit{$\field$-homology manifold} if $\tilde{H}_*(\lk_\Delta \sigma;\field)\cong \tilde{H}_*(\mathbb{S}^{d-1-|\sigma|};\field)$ for every nonempty face $\sigma\in\Delta$. A $(d-1)$-simplicial complex $\Delta$ is \textit{Buchsbaum} over $\field$ if $\Delta$ is pure and for every nonempty face $\sigma$ in $\Delta$, and every $i < d - 1-\dim \sigma$, we have $\tilde{H}_i(\lk_\Delta \sigma;\field)=0$. A $(d-1)$-dimensional simplicial complex $\Delta$ is \textit{Buchsbaum*} over $\field$ if it is Buchsbaum
	over $\field$, and for every pair of faces $\sigma\subseteq \tau$ of $\Delta$, the map $i* : H_{d-1}(\Delta,\cost_\Delta \sigma; \field) \to H_{d-1}(\Delta, \cost_\Delta \tau; \field)$ induced by injection, is surjective. (Here $H_{d-1}(\Delta,\Gamma;\field)$ denotes the relative homology.) A simplicial manifold is a homology manifold as well as a Buchsbaum complex over any field $\field$. An orientable $\field$-homology manifold is Buchsbaum* over $\field$. The following lemma \cite[Theorem 3.1]{BK} provides a basic property of balanced Buchsbaum* complex.
	\begin{lemma}\label{lem: balanced Buchsbaum*}
		Let $\Delta$ be a $(d-1)$-dimensional balanced Buchsbaum* complex. Then the rank-selected subcomplex $\Delta_S$ is Buchsbaum* for every $S\subseteq[d]$.
	\end{lemma}
	\noindent For more properties of balanced Buchsbaum* complexes, see \cite{BK} for a reference.
	
	We will also need some basic facts from homology theory, such as the Mayer-Vietoris sequence, we refer to Hatcher's book \cite{H} as a reference.
	
	\section{The $3d$-vertex Case}
	The main goal of this section is to prove Theorem \ref{thm 3d}, where we verify that the construction of the balanced $3d$-vertex triangulation of $\Sp^{d-2}$-bundle over $\Sp^1$ provided in \cite{KN} is unique. Our result then implies part 1 of Conjecture 6.8 in \cite{KN}. We begin with presenting this construction.
	
	A $d$-dimensional cross-polytope is the convex hull of the set $\{u_1,\cdots,u_d, v_1,\cdots v_d\}$ in $\R^d$, where $u_1,\cdots, u_d$ are $d$ linearly independent vectors in $\mathbb{R}^d$ and $v_i=-u_i$ for $1\leq i\leq d$. The boundary complex of a $d$-dimensional cross-polytope is a balanced $(d-1)$-dimensional sphere with $\kappa(u_i)=\kappa(v_i)=i$ for all $i\in[d]$. Fix integers $n$ and $d$ with $d$ a divisor of $n$, we define a stacked cross-polytopal sphere $\ST^\times(n,d-1)$ by taking the connected sum of $\frac{n}{d}-1$ copies of the boundary complex of the $d$-dimensional cross-polytope. In each connected sum, we identify vertices of the same colors so that $\ST^\times(n,d-1)$ is a balanced $(d-1)$-sphere on $n$ vertices.
	
	From \cite{KN}, we see that there is a balanced simplicial manifold, denoted $BM_d$, with $3d$ vertices that triangulates $\Sp^{d-2}\times \Sp^1$ if $d$ is odd, and triangulates the non-orientable $\Sp^{d-1}$-bundle over $\mathbb{S}^1$ if $d$ is even. This manifold is constructed in the following way: let $\Delta_1$, $\Delta_2$ and $\Delta_3$ be the boundary complexes of $d$-dimensional cross-polytopes with $V(\Delta_1)=\{x_1,\cdots,x_d\}\cup\{y_1,\cdots, y_d\}$, $V(\Delta_2)=\{y_1',\cdots,y_d'\}\cup\{z_1,\cdots,z_d\}$, and $V(\Delta_3)=\{z_1',\cdots,z_d'\}\cup \{x_1',\cdots,x_d'\}$, where each vertex with index $i$ has color $i$. Then $\BM_d$ is exactly the complex we get after forming two connected sums followed by a handle addition that identifies $x_i$, $y_i$, $z_i$ with $x_i'$, $y_i'$, $z_i'$ respectively. Since the number of $(i-1)$-faces of a $d$-dimensional cross-polytope is $2^i\binom{d}{i}$ for $0\leq i\leq d$, it follows immediately that	
	\begin{lemma}\label{lem: f-number of BM_d}
		The number of $(i-1)$-faces of $\mathcal{ST}^\times(n,d-1)$ and $BM_d$ are $[2^i(\frac{n}{d}-1)-(\frac{n}{d}-2)]\binom{d}{i}$ and $3(2^i-1)\binom{d}{i}$, respectively, for $0\leq i\leq d$.
	\end{lemma}
	Now we establish a few other lemmas, the first of which is well-known.
	\begin{lemma}[Alexander Duality]
		Let $\Gamma$ be a triangulation of a homology $(d-1)$-sphere over $\Q$ on vertex set $V$ and $W$ be a subset of $V$. Then $\beta_i(\Gamma[W];\Q)=\beta_{d-i-2}(\Gamma[V-W];\Q)$ for all $i$.
	\end{lemma}
	\begin{lemma}\label{lemma: homology group}
		Let $\Delta$ be a balanced triangulation of a homology $(d-1)$-manifold over $\Q$ ($d\geq 4$), and $W$ be a subset of vertices that all have the same color. Then $\tilde{H}_i(\Delta;\Q)=\tilde{H_i}(\Delta\backslash W;\Q)$ for $1\leq i\leq d-3$.
	\end{lemma}
	\begin{proof}
		Let $v\in W$. Since $\Delta=(\Delta\backslash \{v\})\cup \st_\Delta v$ and $\lk_\Delta v=(\Delta\backslash \{v\})\cap \st_\Delta v$, the Mayer-Vietoris sequence implies that
		\[\cdots \to \tilde{H}_i(\lk_\Delta v;\Q)\to \tilde{H}_i(\Delta\backslash\{ v\};\Q) \oplus \tilde{H}_i(\st_\Delta v;\Q)\to \tilde{H}_i(\Delta;\Q)\to \tilde{H}_{i-1}(\lk_\Delta v;\Q)\to \cdots \ \textrm{is} \ \textrm{exact}.\]
		The complex $\st_\Delta v$ is contractible, so $\tilde{H}_i(\st_\Delta v)=0$ for all $i$. Since $\lk_\Delta v$ is a homology sphere of dimension $d-2$, $\tilde{H}_i(\lk_\Delta v)=0$ for $0\leq i\leq d-3$. Thus
		\[0 \to \tilde{H}_i(\Delta\backslash \{v\};\Q) \to \tilde{H}_i(\Delta;\Q) \to 0 \ \textrm{is} \ \textrm{exact},\]
		which implies that $\tilde{H_i}(\Delta\backslash\{v\};\Q)=\tilde{H}_i(\Delta;\Q)$ for $1\leq i\leq d-3$. Since all vertices in $W$ have the same color, deleting some of them does not change the links of the remaining ones. Therefore the result follows by iterating this argument on other vertices in $W$. 
	\end{proof}
	\begin{lemma}\label{lemma: graph result}
		Let $G_1, G_2, G_3$ be connected graphs on vertex set $U$, where $|U|=2s-1\geq 3$. Further assume that for $\{i,j,k\}=[3]$, every edge of $G_i$ is also an edge of either $G_j$ or $G_k$, and that every $G_i\cap G_j$ has $s$ connected components. Then there exist distinct vertices $u_1,u_2,u_3$ such that the graph $G_i\backslash\{u_i\}$ is disconnected for $i=1,2,3$.
	\end{lemma}
	\begin{proof}
		For $\{i,j,k\}=[3]$, since $G_i\cap G_j$ is a graph on $2s-1$ vertices and it has $s$ connected components, one of the connected components must be a single vertex; we let it be $u_k$. We claim that $u_1,u_2,u_3$ are distinct. Otherwise, assume that $u_1=u_2$. Since every edge of $G_3$ is an edge of either $G_1$ or $G_2$, it follows that $G_3=(G_1\cap G_3)\cup (G_2\cap G_3)$. By the assumption, $\{u_1\}=\{u_2\}$ is a connected component in both $G_2\cap G_3$ and $G_1\cap G_3$. This, however, contradicts the fact that $G_3$ is connected. 
		
		Next consider $G_3\backslash \{u_3\}=\big((G_1\cap G_3)\backslash \{u_3\}\big)\cup \big((G_2\cap G_3)\backslash \{u_3\}\big)$. Since $\{u_3\}$ is not a connected component in either $G_1\cap G_3$ or $G_2\cap G_3$, deleting $u_3$ from these two graphs will not reduce the number of components in the resulting graphs, and hence both $(G_1\cap G_3)\backslash\{u_3\}$ and $(G_2\cap G_3)\backslash\{u_3\}$ have at least $s$ connected components. We claim that $G_3\backslash \{u_3\}$ is disconnected. Indeed, if $G_3\backslash\{u_3\}$ is connected, then there exist at least $s-1$ edges in $(G_2\cap G_3)\backslash\{u_3\}$ so that these edges form a spanning tree on the connected components in $(G_1\cap G_3)\backslash\{u_3\}$. Since $(G_2\cap G_3)\backslash\{u_3\}$ is a graph on $2s-2$ vertices, it implies that the number of connected components in $(G_2\cap G_3)\{u_3\}$ is bounded by $(2s-2)-(s-1)=s-1$, which contradicts the fact that it is at least $s$. Hence, $G_3\backslash\{u_3\}$ is disconnected. Similarly, $G_1\backslash\{u_1\}$ and $G_2\backslash\{u_2\}$ are disconnected.
	\end{proof}
	
	Finally, we quote Theorem 6.6 of \cite{KN}, which will serve as the main tool in proving our theorem.
	\begin{lemma}\label{lemma: theorem 6.6}
		Let $\Delta$ be a balanced triangulation of a homology $(d-1)$-manifold with $\beta_1(\Delta;\Q)\neq 0$.
		\begin{enumerate}
			\item If $d\geq 2$, then $f_{i-1}(\Delta)\geq f_{i-1}(BM_d)$ for all $0<i\leq d$. 
			\item Moreover, if $d\geq 5$, and $(f_0(\Delta),f_1(\Delta),f_2(\Delta))=(f_0(BM_d),f_1(BM_d),f_2(BM_d))$, then $\Delta$ is isomorphic to $BM_d$.
		\end{enumerate} 
	\end{lemma}
	Now we are in a position to prove the main result of this section.
	\begin{theorem}\label{thm 3d}
		If $\Delta$ is a balanced $3d$-vertex triangulation of a homology $(d-1)$-manifold over $\Q$ with $\beta_1(\Delta;\Q)\neq 0$ and $\beta_2(\Delta;\Q)=0$, then $\Delta$ is isomorphic to $\BM_d$.
	\end{theorem}
	\begin{proof}
		 Since $\Delta$ is a homology manifold that is not a homology sphere, $\Delta$ is not a suspension. Therefore, $\Delta$ must have 3 vertices of each color. Since $\Delta$ is a balanced $3d$-vertex homology $(d-1)$-manifold, by part 1 of Lemma \ref{lemma: theorem 6.6} and Lemma \ref{lem: f-number of BM_d}, $f_1(\Delta)\geq f_1(BM_d)=9\binom{d}{2}$. However, since every vertex of $\Delta$ is adjacent to at most $3d-3$ vertices, $f_1(\Delta)\leq 9\binom{d}{2}$. Thus $f_1(\Delta)=f_1(BM_d)=9\binom{d}{2}$, i.e., both of the graphs of $\Delta$ and $BM_d$ are complete $d$-partite graphs. 
		 
		 To prove the theorem, first notice that the cases of $d=3$ and 4 is treated in Proposition 6.9 of \cite{KN} without the assumption $\beta_2(\Delta;\Q)=0$. (In fact, their proposition has an additional assumption that the reduced Euler characteristic of $\Delta$ and $BM_d$ are the same. However, in the case $d=3$, only the condition $f_i(\Delta)=f_i(BM_d)$ for $i=0,1$ is used in their proof; and in the case $d=4$, $\tilde{\chi}(\Delta)=\tilde{\chi}(BM_d)=-1$ holds for any homology 3-manifold $\Delta$.) Now assume that $d\geq 5$. The strategy is to show that $\Delta$ has the same $f_2$ as $BM_d$. The result will then follow from part 2 of Lemma \ref{lemma: theorem 6.6}.
		 
		 We fix some notation here. Given a simplicial complex $\Gamma$, we denote the number of connected components in $\Gamma$ by $c(\Gamma)$ and the graph of $\Gamma$ by $G(\Gamma)$. We let $V_1=\{v_1,v_2,v_3\}$ be the set of vertices of color 1. For every pair $\{i,j\}\subseteq[3]$, set $\Delta_{i,j}:=\lk_\Delta v_i\cap \lk_\Delta v_j$, $\Delta^{i,j}:=\lk_\Delta v_i\cup \lk_\Delta v_j$, and $\Delta_{1,2,3}:=\lk_\Delta v_1\cap \lk_\Delta v_2 \cap \lk_\Delta v_3$. Since all codimension-1 faces of $\Delta$ are contained in exactly two facets of $\Delta$, it follows that $\Delta^{i,j}=\Delta\backslash V_1$, and hence that for $\{i,j,k\}=\{1,2,3\}$,
		\[\Delta_{i,j}\cup \lk_\Delta v_k=(\lk_\Delta v_i\cap \lk_\Delta v_j)\cup \lk_\Delta v_k=\Delta^{i,k}\cap \Delta^{j,k}=\Delta\backslash V_1.\]
		Below all homologies are computed with coefficients in $\Q$. We suppress $\Q$ from our notation. Applying the Mayer-Vietoris sequence to $\Delta\backslash V_1=\lk_\Delta v_i\cup \lk_\Delta v_j$, we obtain
		\begin{multline*}
		\cdots \to \tilde{H}_{m+1}(\lk_\Delta v_i)\oplus \tilde{H}_{m+1}(\lk_\Delta v_j)\to \tilde{H}_{m+1}(\Delta\backslash V_1) \to \\ \tilde{H}_m(\Delta_{i,j})\to \tilde{H}_m(\lk_\Delta v_i)\oplus \tilde{H}_m(\lk_\Delta v_j)\to \cdots.
		\end{multline*}
		If $d\geq 5$, then since all vertex links are $(d-2)$-dimensional homology spheres, $\tilde{H}_2(\lk_\Delta v_i)=\tilde{H}_1(\lk_\Delta v_i)=0$ for all $i$. Taking $m=1$, we conclude that \[\tilde{H}_1(\Delta_{i,j})=\tilde{H}_2(\Delta\backslash V_1) = \tilde{H}_2(\Delta) = 0.\] 
		(The second equality follows from Lemma \ref{lemma: homology group}.) Also taking $m=0$ yields that $\dim\tilde{H}_0(\Delta_{i,j})=\dim\tilde{H}_1(\Delta\backslash V_1) =\dim \tilde{H}_1(\Delta)>0$. Thus $c(\Delta_{i,j})\geq 2$ and it is independent of the pair $i,j$, so we set $s:=c(\Delta_{i,j})$.
		
		Similarly, applying the Mayer-Vietoris sequence to $\Delta\backslash V_1=\Delta_{i,j}\cup\lk_\Delta v_k$, we infer that
		\begin{multline*}
		\cdots \to \tilde{H}_1(\Delta_{i,j})\oplus \tilde{H}_1(\lk_\Delta v_k)\to \tilde{H}_1(\Delta\backslash V_1) \to \tilde{H}_0(\Delta_{1,2,3})\to \\
		\tilde{H}_0(\Delta_{i,j})\oplus \tilde{H}_0(\lk_\Delta v_k)\to \tilde{H}_0(\Delta\backslash V_1)\to\cdots.
		\end{multline*}
		Hence
		\[0\to \tilde{H}_1(\Delta\backslash V_1) \to \tilde{H}_0(\Delta_{1,2,3})\to \tilde{H}_0(\Delta_{i,j})\to 0 \ \textrm{is} \ \textrm{exact},\]
		which implies that $c(\Delta_{1,2,3})= 2s-1\geq 3$.
		
		Since $G(\Delta)$ is a complete $d$-partite graph, for all $1\leq i<j\leq 3$, \[V(\lk_\Delta v_i)=V(\Delta_{i,j})=V(\Delta_{1,2,3})=V(\Delta)\backslash V_1.\]
		Now let $\tilde{G}(\lk_\Delta v_i)$ be the graph obtained from $G(\lk_\Delta v_i)$ by identifying all the vertices in the same connected component in $\Delta_{1,2,3}$ as one vertex. We consider $V(\tilde{G}(\lk_\Delta v_i))$ as the vertex set U (hence each vertex in $U$ represents a connected component in $\Delta_{1,2,3}$) and $\tilde{G}(\lk_\Delta v_i)$ as $G_i$ from Lemma \ref{lemma: graph result}. Since $\tilde{G}(\lk_\Delta v_i)$ and $G(\lk_\Delta v_i)$ are both connected,  and the argument above implies that $G_1, G_2, G_3$ satisfy all the conditions in Lemma \ref{lemma: graph result}, we conclude that there exists a connected component $A_i$ in $\Delta_{1,2,3}$ such that $\tilde{G}(\lk_\Delta v_i)\backslash V(A_i)$ is not connected for $i=1,2,3$. Therefore, the complex $\lk_\Delta v_i\backslash V(A_i)$, whose graph is $G(\lk_\Delta v_i)\backslash V(A_i)$, is also not connected.
		
	    Since $\lk_\Delta v_i$ is a homology $(d-2)$-sphere, by Alexander Duality, \[\beta_0((\lk_\Delta v_i)\backslash V(A_i))=\beta_{d-3}(\lk_\Delta v_i[V(A_i)]),\]
		which implies that $\beta_{d-3}(\lk_\Delta v_i[V(A_i)])$ is also non-zero. Hence $f_0(A_i)\geq d-1$. Since $f_0(A_1\cup A_2\cup A_3)\leq f_0(\Delta_{1,2,3})\leq 3(d-1)$, it follows that $A_1$, $A_2$ and $A_3$ are the only connected components in $\Delta_{1,2,3}$, and each of them has $d-1$ vertices. We obtain that
		\[f_1(\lk_\Delta v_i)\leq \binom{3d-3}{2}-(d-1)^2-2(d-1)=7\binom{d-1}{2},\]
		where the ``$-(d-1)^2$" on the right-hand side comes from the fact that no edges between $A_j$ and $A_k$ exist in $\lk_\Delta v_i$, and ``$-2(d-1)$" comes from the fact that no vertex in $A_i$ can be connected to the other two vertices of the same color. But the lower bound theorem for balanced connected homology manifolds \cite[Theorem 3.2]{KN} implies that $f_1(\lk_\Delta v_i)\geq 7\binom{d-1}{2}$. Hence $f_1(\lk_\Delta v_i)$ is exactly $7\binom{d-1}{2}$ for all $i=1,2,3$. 
		
		Applying the same argument to vertices of other colors, we obtain that for all $v\in V(\Delta)$, $f_1(\lk_\Delta v)=7\binom{d-1}{2}$. Thus
		\[f_2(\Delta)=\frac{1}{3}\sum_{v\in V(\Delta)} f_1(\lk_\Delta v)=21\binom{d}{3},\]
		which, by Lemma \ref{lem: f-number of BM_d}, is the number of 2-faces in $\BM_d$. Then part 2 of Lemma \ref{lemma: theorem 6.6} implies that $\Delta$ is isomorphic to $\BM_d$.
	\end{proof}
	
	\section{The $(3d+1)$-vertex Case}
	The goal of this section is to show that no balanced $(3d+1)$-vertex triangulation of $\Sp^{d-2}$-bundles over $\Sp^1$ exists.  In \cite[Theorem 3.8]{KN}, Klee and Novik proved that any balanced normal pseudomanifold $\Delta$ of dimension $d-1\geq 2$ with $\beta_1(\Delta;\Q)\neq 0$ satisfies $2h_2(\Delta)-(d-1)h_1(\Delta)\geq 4\binom{d}{2}$. Our first step is to show that this result continues to hold for Buchsbaum* complexes. We begin with the following lemma.
	
	\begin{lemma}\label{lem: covering}
		Let $\Delta$ be a Buchsbaum* complex over a field $\field$. If $\Delta$ has a $t$-sheeted covering space $\Delta^t$, then $\Delta^t$ is also  Buchsbaum* over $\field$.
	\end{lemma}
	
	\begin{proof}
		First of all, $\Delta^t$ is Buchsbaum, since $\Delta$ is Buchsbaum and the links in $\Delta^t$ are isomorphic to the links in $\Delta$. For every pair of faces $\sigma^t\subseteq \tau^t$ in $\Delta^t$, their images form a pair of faces $\sigma\subseteq\tau$ in $\Delta$. Let $\hat{\sigma^t}$ and $\hat{\tau^t}$ be the barycenters of $|\sigma^t|$ and $|\tau^t|$ respectively, and let $\hat{\sigma}$ and $\hat{\tau}$ denote their images in $|\sigma|$ and $|\tau|$ respectively. Below we suppress the coefficient field in the homology groups. Consider the following commutative diagram:  
		
		\[
		\xymatrix{
		H_{d-1}(|\Delta^t|,|\Delta^t|-\hat{\sigma^t}) \ar[r]^{\simeq} \ar[d]_{p_*}& H_{d-1}(\Delta^t,\cost_{\Delta^t}\sigma^t)  \ar[r]^{i^t_*} & H_{d-1}(\Delta^t,\cost_{\Delta^t}\tau^t) \ar[r]^{\simeq} & H_{d-1}(|\Delta^t|,|\Delta^t|-\hat{\tau^t}) \ar[d]^{p'_*} \\
		H_{d-1}(|\Delta|,|\Delta|-\hat{\sigma}) \ar[r]^{\simeq} & H_{d-1}(\Delta,\cost_\Delta\sigma) \ar[r]^{i_*} & H_{d-1}(\Delta,\cost_\Delta\tau) \ar[r]^{\simeq} & H_{d-1}(|\Delta|,|\Delta|-\hat{\tau})
	    }
		\]
		
		Since $\Delta$ is Buchsbaum*, the bottom horizontal map $i_*$ is surjective. Also both $p_*$ and $p'_*$ are isomorphisms, since the covering map $p$ is locally an isomorphism. Hence the top horizontal map $i^t_*$ is surjective. Thus by the definition, $\Delta^t$ is Buchsbaum*.
	\end{proof}
	
	\begin{lemma}\label{lem: inequality}
		Let $\Delta$ be a balanced Buchsbaum* (over a field $\field$) complex of dimension $d-1\geq 3$. If $|\Delta|$ has a connected $t$-sheeted covering space, then $2h_2(\Delta)-(d-1)h_1(\Delta)\geq 4\frac{t-1}{t}\binom{d}{2}$. In particular, if $\beta_1(\Delta;\Q)\neq 0$, then $2h_2(\Delta)-(d-1)h_1(\Delta)\geq 4\binom{d}{2}$, or equivalently, $f_1(\Delta)\geq \frac{3(d-1)}{2}f_0(\Delta)$.
	\end{lemma}
	
	\begin{proof}
		The proof follows the same ideas as in \cite[Theorem 4.3]{S1} and \cite[Theorem 3.8]{KN}. Let $X=|\Delta|$ and let $X^t$ be a connected $t$-sheeted covering space of $X$. The triangulation $\Delta$ of $X$ lifts to a triangulation $\Delta^t$ of $X^t$, which is also balanced.
		
		By the previous lemma and Theorem 4.1 in \cite{BK},
		\begin{equation}\label{4.1}
			2h_2(\Delta^t)\geq (d-1)h_1(\Delta^t).
		\end{equation}
		Also by Proposition 4.2 in \cite{S1}, for $i=1,2$,
		\begin{equation}\label{4.2}
			h_i(\Delta^t)=th_i(\Delta)+(-1)^{i-1}(t-1)\binom{d}{i}.
		\end{equation}
		Substituting (\ref{4.2}) for $i=1,2$ in (\ref{4.1}) gives $2h_2(\Delta)-(d-1)h_1(\Delta)\geq 4\frac{t-1}{t}\binom{d}{2}$. The existence of a connected $t$-sheeted covering space of $|\Delta|$ with $\beta_1(\Delta;\Q)\neq 0$ for arbitrary large $t$ implies the in-particular part.
	\end{proof}
	
	The previous lemma implies the following:
		
	\begin{lemma}\label{lem: missing edges}
		If $\Delta$ is a balanced $(3d+1)$-vertex triangulation of $\Sp^{d-2}$-bundle over $\Sp^1$ ($d>3$), whether orientable or non-orientable, then there is a unique color set $W$ containing four vertices, and the graph of $\Delta\backslash W$ is complete $(d-1)$-partite.
	\end{lemma}
	
	\begin{proof}
		The existence of $W$ follows from the same argument as in Theorem \ref{thm 3d}. Assume that $W=\{v_1,v_2,v_3,v_4\}$. First notice that by Lemma \ref{lemma: homology group}, $\beta_1(\Delta;\Q)=\beta_1(\Delta\backslash W;\Q)=1$. Since $\Delta$ is a Buchsbaum* complex over $\Z/2\Z$, by Lemma \ref{lem: balanced Buchsbaum*}, $\Delta\backslash W$ is also Buchsbaum* over $\Z/2\Z$. Thus by Lemma \ref{lem: inequality} and the fact that $\beta_1(\Delta\backslash W;\Q)\neq 0$, it follows that \[f_1(\Delta\backslash W)\geq \frac{3(d-2)}{2}f_0(\Delta\backslash W)=9\binom{d-1}{2}.\]
		However, since every color set in $\Delta\backslash W$ is of cardinality 3, every vertex is connected to at most $3d-6$ vertices in $\Delta\backslash W$. By double counting, \[f_1(\Delta\backslash W)=\frac{1}{2}\sum_{v\in V(\Delta)\backslash W} f_0(\lk_{\Delta\backslash W}v)\leq \frac{(3d-3)(3d-6)}{2}=9\binom{d-1}{2}.\]
		Hence $f_1(\Delta\backslash W)=9\binom{d-1}{2}$ and $f_0(\lk_{\Delta\backslash W}v)=3d-6$ for every vertex $v\in \Delta\backslash W$.  This implies that the graph of $\Delta\backslash W$ is complete $(d-1)$-partite.
	\end{proof}
	\begin{lemma}\label{lem: analog of $3d$-vertex case}
		If $\Delta$ is a balanced $(3d+1)$-vertex triangulation of $\Sp^{d-2}$-bundle over $\Sp^1$ ($d>5$), whether orientable or non-orientable, and $W$ is the unique color set containing four vertices, then $f_1(\lk_{\Delta\backslash W}v)\leq 7\binom{d-2}{2}$ for all $v\in V(\Delta)\backslash W$.
	\end{lemma}
	\begin{proof}
		The proof is very similar to the proof of the crucial upper bound for $f_1(\lk_\Delta v_i)$ in Theorem \ref{thm 3d}. However, since $\Delta\backslash W$ is not a homology manifold, we need to check a few things. Below we use the same notation as in the proof of Theorem \ref{thm 3d}. Given a simplicial complex $\Gamma$, we denote the number of connected components of $\Gamma$ by $c(\Gamma)$. We write $\Delta\backslash W$ as $\bar{\Delta}$ and let $V_1=\{v_1,v_2,v_3\}$ be one color set in $\bar{\Delta}$. For every pair $\{i,j\}\subseteq[3]$, set $\bar{\Delta}_{i,j}:=\lk_{\bar{\Delta}} v_i\cap \lk_{\bar{\Delta}} v_j$,  $\bar{\Delta}^{i,j}:=\lk_{\bar{\Delta}} v_i\cup \lk_{\bar{\Delta}} v_j$ and $\bar{\Delta}_{1,2,3}:=\lk_{\bar{\Delta}} v_1\cap \lk_{\bar{\Delta}} v_2 \cap \lk_{\bar{\Delta}} v_3$. Since all codimension-1 faces are contained in \emph{at least} two facets in $\bar{\Delta}$, $\bar{\Delta}^{i,j}=\bar{\Delta}\backslash V_1$ and $\bar{\Delta}_{i,j}\cup \lk_{\bar{\Delta}} v_k=\bar{\Delta}\backslash V_1$ still holds for $\{i,j,k\}=[3]$. Also for every $v\in\bar{\Delta}$, $\lk_{\bar{\Delta}} v=(\lk_\Delta v)\backslash W$. Hence by Lemma \ref{lemma: homology group}, 
		\begin{equation}\label{eq1}
			\tilde{H}_i(\lk_{\bar{\Delta}} v_i; \Q)=\tilde{H}_i\big((\lk_\Delta v_i)\backslash W;\Q\big)=\tilde{H}_i(\lk_\Delta v_i;\Q)=0
		\end{equation}
		for $i<d-3$. Then applying the Mayer-Vietoris sequence, we obtain that for $i=1,2$, 
		\begin{equation}\label{eq2}
			0=\tilde{H}_i(\lk_{\bar{\Delta}} v_i;\Q)\to \tilde{H}_i(\bar{\Delta}\backslash\{v_i\};\Q) \oplus \tilde{H}_i(\st_{\bar{\Delta}} v_i;\Q)\to \tilde{H}_i(\bar{\Delta};\Q)\to \tilde{H}_{i-1}(\lk_{\bar{\Delta}} v_i;\Q)=0.
		\end{equation}
		(In order for (\ref{eq2}) to hold when $i=2$, it is required that $d>5$.) Since $\st_{\bar{\Delta}} v$ is contractible, by (\ref{eq1}) and (\ref{eq2}) it implies that $\tilde{H}_i(\bar{\Delta}\backslash\{v_i\};\Q)=\tilde{H}_i(\bar{\Delta};\Q)=\tilde{H}_i(\Delta;\Q)$ for $i=1,2$. Iterating the argument on other vertices of $V_1$, it follows that $\tilde{H}_2(\bar{\Delta}\backslash V_1;\Q)=0$ and $\tilde{H}_1(\bar{\Delta}\backslash V_1;\Q)\neq 0$.
		Hence by the proof of Theorem \ref{thm 3d}, we obtain that $c(\bar{\Delta}_{i,j})=s\geq 2$ and $c(\bar{\Delta}_{1,2,3})=2s-1\geq 3$ for every $\{i,j\}\subseteq [3]$.
		
		Next, by Lemma \ref{lem: missing edges}, for every $\{i,j\}\leq 3$ we also have \[V(\lk_{\bar{\Delta}}v_i)=V(\bar{\Delta}_{i,j})=V(\bar{\Delta}_{1,2,3})=V(\bar{\Delta})\backslash V_1.\] Hence applying the same argument that uses Lemma \ref{lemma: graph result} in the proof Theorem \ref{thm 3d}, we conclude that there exist disjoint subcomplexes $A_1, A_2, A_3$ of $\lk_{\bar{\Delta}}v_1, \lk_{\bar{\Delta}}v_2, \lk_{\bar{\Delta}}v_3$ respectively such that $(\lk_{\bar{\Delta}}v_i)\backslash V(A_i)$ is not connected for $i=1,2,3$. However, by Alexander Duality, this implies that
		\[\tilde{\beta_0}\big((\lk_{\bar{\Delta}}v_i)\backslash V(A_i)\big)=\tilde{\beta_0}\big((\lk_{\Delta}v_i)\backslash (V(A_i)\cup W)\big)=\tilde{\beta}_{d-3}(\lk_\Delta v_i[V(A_i)\cup W])\neq 0.\]
		Hence the subcomplex $\lk_\Delta v_i[V(A_i)\cup W]$ is of dimension $\geq d-3$. Since every vertex in $W$ has the same color, it follows that $|V(A_i)|\geq d-3$. However, if $|V(A_i)|=d-3$, then $\lk_\Delta v_i[V(A_i)]$ must be a $(d-4)$-simplex and thus $\tilde{\beta}_{d-3}(\lk_\Delta v_i[V(A_i)\cup W])=0$, a contradiction. So we conclude that $|V(A_i)|\geq d-2$. We proceed using the same argument as in Theorem \ref{thm 3d}, and the result follows.
	\end{proof}
	\begin{lemma}\label{lem: 16-vertex}
		Neither the orientable nor the non-orientable $\Sp^3$-bundle over $\Sp^1$ has a balanced 16-vertex triangulation.
	\end{lemma}
	\begin{proof}
		Assume to the contrary that $\Delta$ is such a triangulation and $V_i$ is the color set for $1\leq i\leq 5$, with $V_5=\{w_1,w_2,w_3,w_4\}$.
		Now take a vertex $u\in V_1$ and let $\Gamma=\lk_\Delta u$.
		
		If $\Gamma\cap V_5=V_5$, then by Lemma \ref{lem: missing edges}, $V(\Gamma)=V(\Delta)\backslash V_1$. Since $\Gamma$ is a 3-sphere and each $\lk_\Gamma w_i$ is a 2-sphere, it follows that
		\begin{equation}\label{4.5}
			f_1(\Gamma)-13=f_1(\Gamma)-f_0(\Gamma)=f_3(\Gamma)=\sum_{i=1}^{4}f_2(\lk_\Gamma w_i)=\sum_{i=1}^{4}(2f_0(\lk_\Gamma w_i)-4).
		\end{equation}
		Take a vertex $v$ of color other than 1 and 5. Since $\lk_\Gamma v$ is a 2-sphere, $\beta_1\big((\lk_\Gamma v)\backslash V_5\big)=|V_5|-1=3$. Hence $(\lk_\Gamma v)\backslash V_5$ cannot be the bipartite graph on six vertices (otherwise its $\beta_1$ is 4), and $f_1(\lk_\Gamma v\backslash V_5)\leq 8$. On the other hand, since every edge of $\lk_\Gamma v\backslash V_5$ is contained in exactly two facets of $\lk_\Gamma v$, it is contained in two of the links $\lk_\Gamma\{v,w_i\}$. Hence $2f_1\big((\lk_\Gamma v)\backslash V_5\big)=\sum_{i=1}^{4} f_1(\lk_\Gamma \{v,w_i\})\geq 16$. This implies $\lk_\Gamma\{v,w_i\}$ is a 4-cycle for every $w_i$ and $v\in V(\Gamma)\backslash W$. Thus $\lk_\Gamma w_i$ is a cross-polytope. By (\ref{4.5}), $f_1(\Gamma)=\sum_{i=1}^{4}(2\cdot 6-4)+13=45$. However, by the lower bound theorem for balanced spheres, $f_1(\Gamma)\geq \frac{1}{2}(9f_0(\Gamma)-4\binom{4}{2})=\frac{93}{2}$, a contradiction. Hence $u$ is not connected to at least one vertex in $V_5$.
		
		Similarly, every vertex in $\cup_{i=1}^{4} V_i$ is not connected to at least one vertex in $V_5$. So there are at least 9 missing edges between the sets $\cup_{i=2}^4 V_i$ and $W$ in $\Delta$. Since $\Delta\backslash V_1$ is Buchsbaum*, by Lemma \ref{lem: inequality},
		\[f_1(\Delta\backslash V_1)\geq \left\lceil\frac{3\cdot 3}{2}f_0(\Delta\backslash V_1)\right\rceil=59.\]
		The complete 4-partite graph on 13 vertices has $63$ edges, so there are no more than 4 missing edges between $\cup_{i=2}^4 V_i$ and $V_5$. This leads to a contradiction and hence no such triangulation exists.

	\end{proof}
	
	We are now ready to state the theorem.
	\begin{theorem}\label{thm 3d+1}
		Neither the orientable nor the non-orientable $\Sp^{d-2}$-bundle over $\Sp^1$ has a balanced $(3d+1)$-vertex triangulation.
	\end{theorem}
	
	\begin{proof}
		The $d=3,4,5$ cases are covered in \cite[Proposition 6.10]{KN} and Lemma \ref{lem: 16-vertex}. Now assume that $d> 5$ and that $\Delta$ is such a triangulation. Let $V_i$ be the set of vertices of color $i$ and let $V_1$ be the unique set of four vertices. By Lemma \ref{lem: analog of $3d$-vertex case}, $f_1(\lk_{\Delta\backslash V_1} v)\leq 7\binom{d-2}{2}$ for all $v\in \Delta\backslash V_1$.
		
		Since $d-2$ divides $f_0(\lk_{\Delta\backslash V_1} v)=3d-6$, by Theorem 4.1 of \cite{BK}, \[f_j(\lk_{\Delta\backslash V_1} v)\geq f_j(\ST^\times(3d-6,d-3))\ \textrm{for} \ \textrm{all} \ j.\]  In particular, by Lemma \ref{lem: f-number of BM_d}, $f_1(\lk_{\Delta\backslash V_1} v)\geq (2^2\cdot 2-1)\binom{d-2}{2}=7\binom{d-2}{2}$. Hence $f_1(\lk_{\Delta\backslash V_1} v)=7\binom{d-2}{2}$. Since $\ST^\times(3d-6,d-3)$ has no missing $k$-faces for $1<k<d-3$, it follows that $f_j(\lk_{\Delta\backslash V_1} v)= f_j(\ST^\times(3d-6,d-3))$ for all $j< d-3$. Thus $\lk_{\Delta\backslash V_1} v$ is either the stacked cross-polytopal sphere or the union of $\ST^\times(3d-6,d-3)$ with its missing facet $\sigma_v$.
		
		On the other hand, the proof of Lemma \ref{lem: analog of $3d$-vertex case} and Theorem \ref{thm 3d} also implies that $\cap_{v\in V_2} \lk_{\Delta\backslash V_1} v$ has three connected components, where each component consists of $d-2$ vertices, all of different colors. Comparing with the structure of $\ST^\times(3d-6,d-3)$, we conclude that these three components are exactly the boundary complexes of the missing facets $\sigma_{v}$ in $\lk_{\Delta\backslash V_1}v$, $v\in V_2$. Thus $\Delta\backslash V_1=\cup_{v\in V_2}\lk_{\Delta\backslash V_1} v$ is the union of $\BM_{d-1}$ with its three missing facets, and hence by Lemma \ref{lem: f-number of BM_d}, $f_{d-2}(\Delta\backslash V_1)=f_{d-2}(BM_{d-1})+3=3\cdot2^{d-1}+3$.
		
		Since for $w\in V_1$, $\lk_\Delta w$ is a homology sphere of dimension $d-2$ as well as a subcomplex of $\Delta\backslash V_1=\cup_{v\in V_2} \lk_{\Delta\backslash V_1} v$, this link is either the cross-polytope or $\ST^\times(3d-3,d-2)$. Thus by Lemma \ref{lem: f-number of BM_d}, $f_{d-2}(\lk_\Delta w) \in\{2^{d-1},2^{d}-1\}$. Therefore,
		\[6\cdot2^{d-1}+6=2f_{d-2}(\Delta\backslash V_1) =\sum_{w\in V_1}f_{d-2}(\lk_\Delta w)=(4+k)2^{d-1}-k,\ \textrm{for} \ \textrm{some} \ k\in\{1,2,3,4\},\]
		where $k$ is the number of vertices $w\in V_1$ such that $f_{d-2}(\lk_\Delta w)=2^d-1$. This leads to a contradiction and shows that no balanced $(3d+1)$-vertex triangulation of $\Sp^{d-2}$-bundle over $\Sp^1$ exists.
	\end{proof}
	
	\begin{remark}
		The same proof also shows that in fact no $\Q$-homology manifold of dimension $d-1\geq 3$ and with $\beta_1(\Delta;\Q)\neq 0$ has a $(3d+1)$-vertex balanced triangulation.
	\end{remark}
	
	\section*{Acknowledgments}
	The author was partially supported by a graduate fellowship from NSF grant DMS-1361423. I would like to acknowledge the invaluable help and guidance from my advisor, Isabella Novik, without whose careful reading and comments it would be impossible to complete this paper. I also thank Steven Klee, Eran Nevo and the referee for pointing out mistakes and giving helpful comments on earlier versions.
	
	\bibliographystyle{amsplain}
	
\end{document}